\documentclass[12pt, reqno]{article}
\usepackage{amsmath, amsthm, amscd, amsfonts, amssymb, graphicx, xcolor}
\usepackage[bookmarksnumbered, colorlinks, plainpages]{hyperref}
\usepackage{tikz}
\usetikzlibrary{arrows,chains,matrix,positioning,scopes,shapes.geometric}
\tikzstyle{arrow} = [thick,->,>=stealth]
\makeatletter
\tikzset{join/.code=\tikzset{after node path={%
\ifx\tikzchainprevious\pgfutil@empty\else(\tikzchainprevious)%
edge[every join]#1(\tikzchaincurrent)\fi}}}
\makeatother
\tikzset{>=stealth',every on chain/.append style={join},
         every join/.style={->}}
\tikzstyle{labeled}=[execute at begin node=$\scriptstyle,
   execute at end node=$]

\theoremstyle{definition}
\newtheorem{theorem}{Theorem}[section]
\newtheorem{lemma}[theorem]{Lemma}
\newtheorem{proposition}[theorem]{Proposition}
\newtheorem{corollary}[theorem]{Corollary}
\newtheorem{definition}[theorem]{Definition}
\newtheorem{example}[theorem]{Example}
\theoremstyle{remark}
\newtheorem{remark}[theorem]{Remark}

\newcommand{\Hom}{{\rm Hom}}
\newcommand{\m}{\mathfrak {m}}
\newcommand{\n}{\mathfrak {n}}
\newcommand{\g}{\mathfrak {g}}
\newcommand{\tal}{\Tilde{\alpha}}
\newcommand{\tr}{\triangleright}
\newcommand{\bcl}{\Bar{\circ}_l}
\newcommand{\bcr}{\Bar{\circ}_r}

\begin{document}
\title{
{On crossed modules of pre-Lie algebras}
 }\vspace{2mm}
\author{  Xuedan Luo\\
\it{Chern Institute of Mathematics, Nankai University},\\ \vspace{2mm} \it{Tianjin, 300071, China}\\
  Email:~  1120230015@mail.nankai.edu.cn}

\date{}

\maketitle

\begin{abstract}
     Given a Lie algebra $\g$ and a $\g$-module $V$, it is due to Gerstenhaber that there is an isomorphism between $H^3(\g,V)$ and the group of equivalence classes of crossed modules with kernel $V$ and cokernel $\g$. The goal of this article is to obtain this result for pre-Lie algebras and their modules.
\end{abstract}
%\tableofcontents

\section{Introduction}
The low degree cohomologies of Lie algebras have some well-known applications. The $1$-cocycles with respect to the adjoint representation define the derivations. %, i.e., $d_{CE}D=0$ if and only if $D[x,y]=[Dx,y]+[x,Dy]$ for $x,y\in\g$ and $D$ a derivation of $\g$. Here $d_{CE}$ is the Chevalley-Eilenberg differential with respect to the adjoint representaion.  
Let $\g$ be a Lie algebra and $V$ be a $\g$-module. Given a $2$-cocycle, one can associate $\g\oplus V$ with a Lie algebra structure to get an exact sequence:
\[0\to V\to\g\oplus V\to\g\to 0,\]
which is called an abelian extension of $\g$ by $V$. This association induces an bijection between $H^2(\g,V)$ and the set of equivalence classes of abelian extensions, so the second cohomology classifys the abelian extensions. 

The notion of crossed module was proposed by Whitehead \cite{Wh} in the context of homotopy theory, which consists of a homomorphism $\mu:\m\to\n$ of Lie algebras and a compatible action of $\n$ on $\m$. We attribute to Gerstenhaber \cite{G 1964, G 1966} the theorem stating that the group of equivalence classes of crossed modules with kernel $V$ and cokernel $\g$, is isomorphic to $H^3(\g,V)$. One may express a crossed module with kernel $V$ and cokernel $\g$ by the following exact sequence:
\[0\to V\to\m\xrightarrow{\mu}\n\to\g\to 0,\]
which is called a crossed module extension of $\g$ by $V$.
So roughly, we may restate the result of Gerstenhaber by: the third cohomology classifys the crossed module extensions. The original proof is quite abstract using the $Ext$ functor, and we recommend \cite{W} for a more constructive and algebraic proof.

The goal of this article is to get the version of this result in pre-Lie algebras. For versions of this result in Lie $2$-algebras, Lie-Rinehart algebras, pre-Lie Rinehart algebras and zinbiel algebas, see \cite{LL,CLL,CLP,Z}.

The notion of pre-Lie algebra was introduced by Gerstenhaber \cite{G 1966} in the study of deformations of associative algebras.
A pre-Lie algebra is a vector space $\g$ equipped with a map $\circ:\g\otimes\g\to\g$ such that the associator is left-symmetry, i.e.,
$(x,y,z)=(y,x,z)$, where $(x,y,z):=(x\circ y)\circ z-x\circ (y\circ z)$ for $x,y,z\in\g$.
  It arises naturally from many areas, such as the left-invariant affine structures on Lie groups, complex structures on Lie algebras and so on, see \cite{B,Burde}. Its name is due to its closed relation to the Lie algebras, i.e., the commutator defines a Lie algebra on the underlying space. Moreover, the left multiplication of pre-Lie algebra gives a representation of the commutator Lie algebra. %The cohomology of pre-Lie algebras was developed in \cite{D}, which is also closely associated with the cohomology of Lie algebras. That is, there is an isomorphism between the $n$-th cohomology of pre-Lie algebras and the $(n-1)$-th cohomology of its commutator Lie algebra, see \cite{GLST}.

%Based on the notion of crossed modules of pre-Lie algebras given in \cite{S}, we prove the pre-Lie version of Gerstenhaber's result: crossed modules of pre-Lie algebras can be classified via the third cohomology (Theorem \ref{m th}). Moreover, by the strong link between cohomologies of pre-Lie algebras $\g$ and its commutator Lie algebras $\g^c$, we deduce a bijection between equivalence classes of crossed module extensions of $\g$ and equivalent classes of abelian extensions of $\g^c$ (Corollary \ref{{cor of m th}}).

The paper is organized as follows. Section $2$ constitutes sufficient preparation.  Section $3$ states and proves our main theorem (Theorem \ref{m th}) and get a corollary (Corollary \ref{cor of m th}). Section $4$ summarizes some relations between crossed modules of pre-Lie algebras and crossed modules of Lie, Rota-Baxter Lie, and dendriform algebras.

\section{Preliminaries}
A {\bf pre-Lie algebra} $(\g,\circ)$ is a vector space $\g$ equipped with a bilinear map $\circ:\g\otimes \g\to \g$ such that 
\begin{equation}\label{defi: pre}
    (x\circ y)\circ z-x\circ (y\circ z)=(y\circ x)\circ z-y\circ (x\circ z),\qquad\forall x,y,z\in \g.
\end{equation}
 Let $(\g,\circ)$ be a pre-Lie algebra. Then there is a Lie algebra structure $[\cdot,\cdot]:\g\wedge \g\to \g$ on $\g$ defined by $[x,y]:=x\circ y-y\circ x$. We call $(\g,[\cdot,\cdot])$ the sub-adjacent Lie algebra of $(\g,\circ)$ and denote it by $\g^c$.
 \begin{definition}
    A {\bf representation} of a pre-Lie algebra $(\g,\circ)$ is a pair $(V,\circ_l,\circ_r)$, where $V$ is a vector space, $\circ_l:\g\otimes V\to V$ is a representation of the sub-adjacent Lie algebra $\g^c$ on $V$, and $\circ_r:V\otimes \g\to V$ is a bilinear operation satisfying 
    \begin{eqnarray}
        (x\circ_l u)\circ_r y-x\circ_l (u\circ_r y)&=&(u\circ_r x)\circ_r y-u\circ_r (x\circ y),\nonumber\qquad\forall x,y\in\g,u\in V.
    \end{eqnarray}
   \end{definition}
   In this case, we also say that $V$ is a $\g$-module. There is a natural representation $(\g,\circ_L,\circ_R)$ of $\g$ on itself (or on its two-sided ideal $R$, i.e., a subspace such that $R\circ\g\subset R,\ \g\circ R\subset R$) defined by $x\circ_L y=x\circ y,\ y\circ_R x=y\circ x$.  
\begin{definition}
    Given two pre-Lie algebras $(\m,\circ_m)$ and $(\n,\circ_n)$, an {\bf action} of $\n$ on $\m$ is a representation $(\circ_l,\circ_r)$ of $\n$ on $\m$ satisfying $2$ more equalities:
    \begin{eqnarray*}
        (x\circ_l u)\circ_m v-x\circ_l (u\circ_m v)&=&(u\circ_r x)\circ_m v-u\circ_m (x\circ_l v),\\
        (u\circ_m v)\circ_r x-u\circ_m (v\circ_r x)&=&(v\circ_m u)\circ_r x-v\circ_m (u\circ_r x),
    \end{eqnarray*}
    for all $x\in\n,u,v\in\m$.
\end{definition}
We see that the action reduces to a representation when the pre-Lie structure on $\m$ is trivial.

The cohomology of pre-Lie algebra $(\g,\circ)$ with a representation $(V,\circ_l,\circ_r)$ is given in \cite{D}. %see also \cite{GLST,MSW}. 
The set of $n$-cochains is 
\[C^n(\g,V)=\Hom (\wedge^{n-1}\g\otimes\g,V),\qquad n\geq 1.\] 
The coboundary operator $d:C^n(\g,V)\to C^{n+1}(\g,V)$ is defined by
\begin{eqnarray*}
    (df)(x_1,\cdots,x_{n+1}) &=&\sum\limits_{i=1}^n (-1)^{i+1}x_i\circ_l f(x_1,\cdots,\hat{x_i},\cdots,x_{n+1})
    \\
    &&+\sum\limits_{i=1}^{n} (-1)^{i+1}f(x_1,\cdots,\hat{x_i},\cdots,x_n,x_i)\circ_r x_{n+1}\\
    &&-\sum\limits_{i=1}^n(-1)^{i+1}f(x_1,\cdots,\hat{x_i},\cdots,x_n,x_i\circ x_{n+1})\\
    &&+\sum\limits_{1\leq i< j\leq n}(-1)^{i+j}f([x_i,x_j],x_1,\cdots,\hat{x_i},\cdots,\hat{x_j},\cdots,x_{n+1}),
\end{eqnarray*}
for $x_1,\ldots,x_{n+1}\in\g, f\in C^n(\g,V)$.
Here $[x_i,x_j]:=x_i\circ x_j-x_j\circ x_i$.
In particular, for $n=1$, $(df)(x,y)=x\circ_l f(y)+f(x)\circ_r y-f(x\circ y)$ and for $n=2$,
\begin{eqnarray*}     
    (df)(x,y,z)&=& x\circ_l f(y,z)-y\circ_l f(x,z)+f(y,x)\circ_r z-f(x,y)\circ_r z\\&&
    -f(y,x\circ z)+f(x,y\circ z)-f([x,y],z).
\end{eqnarray*}
\begin{definition}
    Let $(\m,\circ_m),(\n,\circ_n)$ be pre-Lie algebras with an action $(\circ_l,\circ_r)$ of $\n$ on $\m$. A {\bf crossed module} is a homomorphism $\mu:\m\to \n$ of pre-Lie algebras satisfying
    \begin{align*}
        \mu(u\circ_r x)&=\mu(u)\circ_n x,\qquad \mu(x\circ_l u)=x\circ_n \mu(u),\\
        \mu(u)\circ_l v&=u\circ_m v=u\circ_r \mu(v),\qquad \forall x\in \n,\ u,v\in \m,
    \end{align*}
    denoted by $(\m\xrightarrow{\mu}\n,\circ_l,\circ_r)$.
\end{definition}
\begin{example}
    \begin{enumerate}
        \item[1.] Let $(\n,\circ)$ be a pre-Lie algebra, $R$ be its two-sided ideal and $i$ be the inclusion. Then $(R\xrightarrow{i}\n,\circ_L,\circ_R)$ is a crossed module. In particular, $(\n\xrightarrow{id}\n,\circ_L,\circ_R)$ is a crossed module.
        \item[2.] Let $f:\m\to \n$ be a homomorphism of pre-Lie algebras. Then $ker(f)$ is a two-sided ideal of $(\m,\circ_m)$, thus by the example above,  $(ker(f)\xrightarrow{i}\m,\circ_{L},\circ_R)$ is a crossed module. 
        \item[3.] Let $(V,\circ_l,\circ_r)$ be a representation over pre-Lie algebra $\g$. Endow $V$ with the trivial pre-Lie structure, then $(V\xrightarrow{0}\g,\circ_l,\circ_r)$ is a crossed module.
    \end{enumerate}
\end{example}
Let ${\bf preLie}$ be the category whose objects are pre-Lie algebras, morphisms are homomorphisms of pre-Lie algebras.
\begin{definition}
    Let $(\g,\circ)$ be a pre-Lie algebra, and let $(V,\circ_l,\circ_r)$ be a representation of it. A {\bf crossed module extension} of $\g$ by $V$ is an exact sequence in category {\bf preLie}:
    \begin{equation}\label{cmext}
        0\to V\xrightarrow{i}\m\xrightarrow{\mu}\n\xrightarrow{\pi}\g\to 0
    \end{equation}
    such that $\m\xrightarrow{\mu}\n$ is a crossed module. 
\end{definition}

\begin{remark}\label{re mo}
Given an exact sequence \eqref{cmext} such that $\m\xrightarrow{\mu}\n$ is a crossed module, the action of $\n$ on $\m$ uniquely induces a representation of $\g$ on $V$, and we denote them by the same notation $(\circ_l,\circ_r)$. Choose any linear section $\rho:\g\to\n$ of $\pi$ such that $\pi\rho=id$, this representation is given by
\[i(x\circ_l u):=\rho(x)\circ_l i(u),\quad i(u\circ_r x):=i(u)\circ_r\rho(x),\quad\forall x\in\g,u\in V.\]
It is independent of the choice of $\rho$ because if $\rho_1$ and $\rho_2$ are two sections of $\pi$, then $(\rho_1-\rho_2)\in ker\pi=im\mu$, thus $i(u)\circ_r(\rho_1-\rho_2)(x)=i(u)\circ_r\mu(m)$ for some $m\in\m$. By the crossed module properties, this equals $i(u)\circ_m m$, which vanishes since $i(V)$ is the center of $\m$ (say $i(v)\circ_m m=\mu(i(v))\circ m=0, m\in\m,v\in V$). Similar argument applys to $\rho(x)\circ_l i(u)$.
\end{remark}

\begin{example}
    Given a crossed module $\m\xrightarrow{\mu}\n$, we get a natural crossed module extension \eqref{cmext}, where $\g=coker(\mu)=\n/im\mu$ and $V=ker(\mu)$ as vector spaces. By the crossed module properties, $im\mu$ is an ideal of $\n$, therefore $\g$ has an induced pre-Lie structure from $\n$. By Remark \ref{re mo}, $V$ is indeed a representation of $\g$. 
\end{example}

\begin{definition}\label{eq cm ext}
    Two crossed module extensions of $\g$ by $V$ are {\bf elementary equivalent}, denoted by $E\hookrightarrow E'$, if there exists homomorphisms $r$ and $s$ of pre-Lie algebras such that 
    \begin{itemize}
        \item the following diagram commutes:
    \\
\begin{center}
\begin{tikzpicture}
  \matrix (m) [matrix of math nodes, row sep=3em, column sep=3em]
    { E:& 0 & V  & \m & \n & \g & 0 \\
      E':& 0 & V & \m' & \n' & \g & 0; \\ };
  { [start chain] \chainin (m-1-2);
    \chainin (m-1-3);
    { [start branch=V] \chainin (m-2-3)
        [join={node[right,labeled] {id}}];}
    \chainin (m-1-4) [join={node[above,labeled] {i}}];
    { [start branch=m] \chainin (m-2-4)
        [join={node[right,labeled] {r}}];}
    \chainin (m-1-5) [join={node[above,labeled] {\mu}}];
    { [start branch=n] \chainin (m-2-5)
        [join={node[right,labeled] {s}}];}
    \chainin (m-1-6) [join={node[above,labeled] {\pi}}]; 
    { [start branch=g] \chainin (m-1-7);}
  \chainin (m-2-6) [join={node[right,labeled] {id}}];}
  { [start chain] \chainin (m-2-2);
    \chainin (m-2-3);
    \chainin (m-2-4) [join={node[above,labeled] {i'}}];
    \chainin (m-2-5) [join={node[above,labeled] {\mu'}}];
    \chainin (m-2-6) [join={node[above,labeled] {\pi'}}]; 
    \chainin (m-2-7);}
\end{tikzpicture}
\end{center}
\item they respects the actions in the crossed modules. That is, $r(u\circ_r x)=r(u)\circ_r s(x)$ and $r(x\circ_l u)=s(x)\circ_l r(u)$ for $x\in\n,u\in\m$.
    \end{itemize}
\end{definition}
     It is easy to see that the relation $\hookrightarrow$ is not symmetric for $n\geq 2$. However, it generates an equivalence relation. We define crossed module extensions $E$ and $E'$ to be {\bf equivalent}, denoted by $E\sim E'$, if there exists a chain $E_1,\ldots,E_k$ such that \[E\hookrightarrow E_1\hookleftarrow E_2\hookrightarrow\cdots\hookrightarrow E_k\hookleftarrow E'. \] 
The set of equivalence classes of crossed module extensions of $\g$ by $V$, which induce the given representation of $\g$ on $V$, is denoted by $CExt(\g,V)$.
\begin{remark}
    We remark that if two extensions are elementary equivalent, then they are equivalent. Suppose that $E\hookrightarrow E'$, then there exists a crossed module extension
    \[E_1:\ 0\to V\to\m/ker(r)\xrightarrow{\Tilde{\mu}}\n/ker(s)\to\g\to 0,\]
    where $\m/ker (r)\xrightarrow{\Tilde{\mu}}\n/ker(s)$ is a crossed module with the induced homomorphism $\Tilde{\mu}$ and action. We have $E\sim E'$ bacause $E\hookrightarrow E_1$ through the projections and $E_1\hookleftarrow E'$ through the isomorphisms $\Tilde{r},\Tilde{s}$ induced by $r,s$. 
    
    We check that $ker\Tilde{\mu}=V$ and $coker\Tilde{\mu}=\g$, i.e., $ker\Tilde{\mu}=\Tilde{r}^{-1}ker(\mu')=\Tilde{r}^{-1}V=V$ since $\Tilde{r}$ is an isomorphism with $\Tilde{r}|_V=id_V$, $coker\Tilde{\mu}=(\n/ker(s))/\Tilde{\mu}(\m/ker(r))=(\n/ker(s))/(\mu\m/ker(s))\cong\n/\mu\m=\g$. 
\end{remark}
\section{Classification of crossed modules via $H^3$}
In this section, we proves our main theorem (Theorem \ref{m th}), which says crossed modules of pre-Lie algebras with kernel $V$ and cokernel $\g$ can be classified by $H^3(\g,V)$. As an immediate consequence, we obtain Corollary \ref{cor of m th}, which shows the relation between crossed module extension of pre-Lie algebra and abelian extension of its sub-adjacent Lie algebra.
\begin{lemma}\label{free}
    Let $F$ be a free pre-Lie algebra, $V$ be an $F$-module. Then $H^n(F,V)=0$ for $n\geq 2$.
\end{lemma}
\begin{proof}
    Denote by $F^c$ the sub-adjacent Lie algebra of $F$. The $F$-module $(V,\circ_l,\circ_r)$ can be viewed as an $F^c$-module via $[X,v]:=X\circ_l v-v\circ_r X$ for $X\in F^c,v\in V$. Then $\Hom(F^c,V)$ is also an $F^c$-module with the representation $\rho:F^c\times\Hom(F^c,V)\to\Hom(F^c,V)$ given by
    $\rho(X)f(Y):=[X,f(Y)]-f([X,Y]_{F^c})$,for $X,Y\in F^c,f\in\Hom(F^c,V)$.
    By Theorem $3.4$ of \cite{D}, we have 
    \[H_{preLie}^{n}(F,V)\cong H_{Lie}^{n-1}(F^c,\Hom(F^c,V)),\qquad n\geq 1.\]
    By Corollary $6.3$ of \cite{Ch}, the sub-adjacent Lie algebra of free pre-Lie algebra is also free, thus $F^c$ is free.
    Now, as a standard result in homological algebras (see \cite{HS}), 
    \[H^n_{Lie}(F^c,\Hom(F^c,V))=0,\qquad n\geq 2.\] Thus $H_{preLie}^{n}(F,V)=0$ for $n\geq 3$. For $n=2$, let $Ext(F,V)$ be the set of equivalence class of abelian extensions of $F$ by $V$, which is zero since $F$ is free. By the fact that $H^2_{preLie}(F,V)\cong Ext(F,V)$, we obtain $H^2_{preLie}(F,V)=0$. The proof is complete.
\end{proof}
\begin{theorem}\label{m th}
   Let $(\g,\circ_{\g})$ be a pre-Lie algebra and let $(V,\circ_l,\circ_r)$ be a representation of $\g$. There is a bijection between $CExt(\g,V)$ and $H^3(\g,V)$.
\end{theorem}
\begin{proof}
   Since the proof is long, we only make precise here how to construct the bijection and its inverse, and details will be left to Lemma \ref{1} to \ref{4}.
   
   $\bullet$ We construct $T:CExt(\g,V)\to H^3(\g,V)$ as follows. Given a crossed module extension which induces the given representation of $\g$ on $V$:
\begin{equation}\label{ext}
    0\to V\xrightarrow{i}\m\xrightarrow{\mu}\n\xrightarrow{\pi}\g\to 0,
\end{equation}
where the action of $\n$ on $\m$ is denoted by $(\Bar{\circ}_l,\Bar{\circ}_r)$.
The first step is to take a linear section $\rho$ of $\pi$ such that $\pi\rho=id$, and to calculate the default of $\rho$ to be a pre-Lie homomorphism. That is,  $\alpha(x,y):=\rho(x)\circ_{\n}\rho(y)-\rho(x\circ_{\g} y)$ for $x,y\in\g$.
 Since $\pi$ is a pre-Lie homomorphism, we have $\pi\alpha(x,y)=0$. By the exactness of sequence \eqref{ext}, there exists $\beta\in \Hom(\g\otimes\g,\m)$ such that $\mu\beta=\alpha$.
 Choosing a linear section $\sigma$ of $\mu$ such that $\mu\sigma=id$, we can take $\beta$ as $\sigma\alpha$.
 
Next, we define $\theta\in C^3(\g,\m)$ by \begin{align}\label{theta1}
            \theta (x,y,z)
            &=\rho(x)\bcl\beta(y,z)-\rho(y)\bcl\beta(x,z)+\beta(y,x)\bcr \rho(z)-\beta(x,y)\bcr \rho(z)\nonumber\\
            &\ \ -\beta(y,x\circ_{\g} z)+\beta(x,y\circ_{\g} z)
            -\beta([x,y],z),           
        \end{align}
        where $[x,y]=x\circ_{\g} y-y\circ_{\g} x$. 
        Despite $\theta$ has the form of $d\beta$, we can not say $\theta=d\beta$ because the map $\g\otimes\m\to\m,(x,m)\mapsto\rho(x)\circ m$ is not a representation in general. But in case $\beta\in \Hom(\g^{\otimes 2},i(V))$, this is just the given representation of $\g$ on $V$.
    
Finally, by the crossed module properties and $\alpha=\mu\beta$, we compute $\mu\theta=0$, thus $\theta\in C^3(\g,V)$.
It is straightforward to check that $\theta$ is a $3$-cocycle, thus we can define $T:CExt(\g,V)\to H^3(\g,V)$ by mapping the equivalence classs of extension \eqref{ext} to $[\theta]$. 

$\bullet$ Conversely, the inverse map $S:H^3(\g,V)\to CExt(\g,V)$ is constructed by the following procedures. 
First, let $F$ be the free pre-Lie algebra generated by $\g$ (see \cite{D}), then we get an exact sequence:
\begin{equation}\label{eq 3}
    0\to K\xrightarrow{i} F\xrightarrow{\pi} \g\to 0,
\end{equation}
where $\pi$ is the projection, $K=ker(\pi)$ and $i$ is the inclusion. 

Second, $\pi$ induces a map $\pi^*:C^n(\g,V)\to C^n(F,V)(n\geq 1)$ by $f\mapsto f\pi$. Regarding the $\g$-module $V$ as an $F$-module and also trivial $K$-module by $X\circ u:=\pi(X)\circ u$ for $X\in F, u\in V$, then $\pi^*$ is a cochain map of complexes. That is, the following diagram commutes: 
\begin{equation}
\begin{array}{ccccccccc}
\cdots&\longrightarrow& C^2(\g,V)&\stackrel{d}\longrightarrow&C^3(\g,V)&\stackrel{d}\longrightarrow&C^4(\g,V)&\longrightarrow&\cdots\\
 &            &\pi^*\Big\downarrow&       &\pi^*\Big\downarrow&          &\pi^*\Big\downarrow& &\\
 \cdots&\longrightarrow&C^2(F,V)&\stackrel{d}\longrightarrow&C^3(F,V)&\stackrel{d}\longrightarrow&C^4(F,V)&\longrightarrow&\cdots.
 \end{array}\end{equation}
%\\
%\begin{center}
%\begin{tikzpicture}
    %\matrix (m) [matrix of math nodes, row sep=3em, column sep=3em]
    %{ \cdots & C^2(\g,V) & C^3(\g,V) & C^4(\g,V) & \cdots \\
     % \cdots & C^2(F,V) & C^3(F,V) & C^4(F,V) & \cdots. \\ };
  %{ [start chain] \chainin (m-1-1);
   % \chainin (m-1-2);
    %{ [start branch=C^2(g,V)] \chainin (m-2-2)
     %   [join={node[right,labeled] {\pi^*}}];}
    %\chainin (m-1-3) [join={node[above,labeled] {d}}];
    %{ [start branch=C^3(g,V)] \chainin (m-2-3)
     %   [join={node[right,labeled] {\pi^*}}];}
    %\chainin (m-1-4) [join={node[above,labeled] {d}}];
    %{ [start branch=C^4(g,V)] \chainin (m-2-4)
     %   [join={node[right,labeled] {\pi^*}}];}
    %\chainin (m-1-5);}
  %{ [start chain] \chainin (m-2-1);
   % \chainin (m-2-2);
   % \chainin (m-2-3) [join={node[above,labeled] {d}}];
   % \chainin (m-2-4) [join={node[above,labeled] {d}}]; 
   % \chainin (m-2-5);}
%\end{tikzpicture}
%\end{center}
Likewise, the inclusion $i:K\to F$ also induces a cochain map $i^*:C^n(F,V)\to C^n(K,V)$.
For any $3$-cocycle $\theta\in C^3(\g,V)$, $\pi^*\theta$ is a $3$-cocycle because $d\pi^*\theta=\pi^*d\theta=0$. By Lemma \ref{free}, $H^n(F,V)=0\ (n>1)$ for free pre-Lie algebra $F$, %(see \cite{CLL,CL}), 
there exists an $\Tilde{\alpha}\in C^2(F,V)$ such that $d\Tilde{\alpha}=\pi^*\theta$.
We thus get a map $\alpha:=i^*\tal\in C^2(K,V)$, which is a $2$-cocycle because $d\alpha=di^*\tal=i^*d\tal=i^*\pi^*\theta=0$. 

Now, let $\m=K\oplus V$ as vector space and define an operation $\circ_{\m}:\m\otimes\m\to\m$ by 
\begin{equation}\label{pre lie m}
    (X+u)\circ (Y+v):=X\circ_F Y+\alpha(X,Y),
\end{equation} for $X,Y\in K, u,v\in V$. Since $\alpha$ is a $2$-cocycle, $(\m,\circ_{\m})$ is a pre-Lie algebra. We thus get an exact sequence of pre-Lie algebras:
\begin{equation}\label{eq 4}
    0\to V\xrightarrow{i} \m\xrightarrow{p} K\to 0,
\end{equation}
where $i$ is the inclusion and $p$ is the projection.

Finally, splicing \eqref{eq 4} and \eqref{eq 3} together, we get a crossed module extension:
\begin{equation}\label{eq 5}
    0\to V\xrightarrow{i}\m\xrightarrow{\mu}F\xrightarrow{\pi}\g\to 0,
\end{equation}
where $\mu(X+u)=X,X+u\in K\oplus V=m$. The action $(\bcl,\bcr)$ of $F$ on $\m$ is given by
\begin{eqnarray}\label{tal action}
    X\circ (Y+v) &:=X\circ_F Y+\tal(X,Y)+X\circ_l v,\nonumber\\
    (Y+v)\circ X &:=Y\circ_F X+\tal(Y,X)+v\circ_r X,
\end{eqnarray}
for $X\in F,\ Y+v\in K\oplus V=\m$. This is an action because $d\tal(X,Y,Z)=\pi^*\theta(X,Y,Z)=\theta(\pi X,\pi Y,\pi Z)=0$ when one variable is in $K=ker\pi$, and $K$ acts trivially on $V$. We define the map $S: H^3(\g,V)\to CExt(\g,V)$ by mapping $[\theta]$ to equivalence class of crossed module extension \eqref{eq 5}.

$\bullet$ Lemma \ref{1} to \ref{4} below will show that $T$ and $S$ are well-defined and inverse to each other. Therefore, $T$ is a bijection and the proof is complete . 
\end{proof}

\begin{lemma}\label{1}
    The above-mentioned $T: CExt(\g,V)\to H^3(\g,V)$ is well-defined.
\end{lemma} 
\begin{proof}
     We follow the notations in the proof of Theorem \ref{m th}.
     We shall check that $[\theta]$ is independent of the choice of sections $\sigma$ and $\rho$, and that equivalent extensions lead to $3$-cocycles in the same cohomology class.
     
     $\bullet$ For two sections $\sigma$ and $\sigma'$ of $\mu$, we have $\mu(\sigma-\sigma')=0$, which implies $\sigma-\sigma'\in\Hom(\n,V)$ and thus $\beta-\beta'=(\sigma-\sigma')\alpha\in \Hom(\g\otimes\g,V)$. So $\theta-\theta'=d^2(\beta-\beta')$ and we have $[\theta-\theta']=0$.
         
      $\bullet$ To check that $[\theta]$ is independent of the choice of $\rho$, let $\rho$ and $\rho'$ be two sections of $\pi$, then they induce $\beta,\beta'\in \Hom(\g\otimes\g,\m)$, and thus $\theta,\theta'\in C^3(\g,V)$. Since $\pi(\rho-\rho')=0$ and \eqref{ext} is exact, there exists $h\in \Hom(\g\otimes\g,\m)$ such that $\rho-\rho'=\mu h$.
      To use this relation, we would like to express $\theta-\theta'$ in terms of $\rho-\rho'$. 
      Explicitly, since 
        \begin{align*}
            \rho(x)\bcl\beta(y,z)-\rho'(x)\bcl\beta'(y,z)
            &=(\rho-\rho')(x)\bcl\beta(y,z)+\rho'(x)\bcl(\beta-\beta')(y,z)\\
            &=\mu h(x)\bcl\sigma\alpha(y,z)+\rho'(x)\bcl(\beta-\beta')(y,z)\\ &=h(x)\bcr\mu\sigma\alpha(y,z)+\rho'(x)\bcl(\beta-\beta')(y,z)\\       &=h(x)\bcr\alpha(y,z)+\rho'(x)\bcl(\beta-\beta')(y,z)\\
            &=h(x)\bcr(\rho(y)\circ_{\n}\rho(z)-\rho(y\circ_{\g} z))\\
            &\ \ +\rho'(x)\bcl(\beta-\beta')(y,z),
        \end{align*}
        and similarly
        \begin{align*}
            \beta(y,x)\bcr\rho(z)-\beta'(y,x)\bcr\rho'(z)&=(\rho(y)\circ_{\n}\rho(x)-\rho(y\circ_{\g} x))\circ h(z)\\
            &\ \ +(\beta-\beta')(y,x)\bcr\rho'(z),
        \end{align*}
        we obtain
        \begin{align}\label{theta}
            (\theta-\theta')(x,y,z)&=\rho(x)\bcl\beta(y,z)-\rho'(x)\bcl\beta'(y,z)\nonumber\\
            &\ \ +\beta(y,x)\bcr\rho(z)-\beta'(y,x)\bcr \rho'(z)\nonumber\\
            &\ \ -\rho(y)\bcl\beta(x,z)+\rho'(y)\bcl\beta'(x,z)\nonumber\\
            &\ \ -\beta(x,y)\bcr \rho(z)+\beta'(x,y)\bcr \rho'(z)\nonumber\\
            &\ \ -\beta(y,x\circ_{\g} z)+\beta(x,y\circ_{\g} z)
            -\beta([x,y],z)\nonumber\\
            &\ \ +\beta'(y,x\circ_{\g} z)-\beta'(x,y\circ_{\g} z)
            +\beta'([x,y],z)\nonumber\\
           &=\{h(x)\bcr(\rho(y)\circ_{\n}\rho(z)-\rho(y\circ_{\g} z))+\rho'(x)\bcl(\beta-\beta')(y,z)\nonumber\nonumber\\
            &\ \ +(\rho(y)\circ_{\n}\rho(x)-\rho(y\circ_{\g} x))\bcl h(z)+(\beta-\beta')(y,x)\bcr\rho'(z) \nonumber\\
            &\ \ -(\beta-\beta')(y,x\circ_{\g} z)-c.p.(x,y)\}-(\beta-\beta')([x,y],z),
         \end{align}
         for $x,y,z\in\g$.
         %Consider $\mu(\beta-\beta')(x,y)$ and 
         Define $b\in\Hom(\g\otimes\g,\m)$ by $b(x,y):=\rho(x)\bcl h(y)-h(x)\bcr\mu h(y)+h(x)\bcr\rho(y)-h(x\circ_{\g} y)$ such that $\mu b=\mu(\beta-\beta')$, then $\beta-\beta'-b\in\Hom(\g\otimes\g,V)$. So 
         \begin{align*}
             \eqref{theta}&=\{h(x)\bcr(\rho(y)\circ_{\n}\rho(z)-\rho(y\circ_{\g} z))+\rho'(x)\bcl b(y,z)\\
            &\ \ +(\rho(y)\circ_{\n}\rho(x)-\rho(y\circ_{\g} x))\bcl h(z)+b(y,x)\bcr\rho'(z)-b(y,x\circ_{\g} z)-c.p.(x,y)\}\\
            &\ \ -b([x,y],z)+d(\beta-\beta'-b),
         \end{align*}
         which means that replacing $\beta-\beta'$ by $b$ would not change the cohomology class of $\theta-\theta'$. This replacement yields $[\theta-\theta']=[0]$ (this process requires comprehensive application of the definitions of the pre-Lie algebra $\g$, the action of $\n$ on $\m$, and the crossed module properties).

         $\bullet$ It remains to show that equivalent extensions lead to $3$-cocycles in the same cohomology class, which is sufficient to show for elementary equivalent extensions. Given elementary equivalent crossed module extensions $E\hookrightarrow E'$ (see definition \ref{eq cm ext}), take sections $\rho,\rho'$ of $\pi,\pi'$ and sections $\sigma,\sigma'$ of $\mu,\mu'$, respectively. We would like to prove $[\theta-\theta']=0$. Since $[\theta']$ is independent of the choice of section, we can take $\rho'=s\rho$, then for $x,y,z\in\g$,
         \begin{align}\label{theta-theta'}
             (\theta-\theta')(x,y,z)&=(r\theta-\theta')(x,y,z)\nonumber\\
             &=\{s\rho(x)\bcl(r\sigma-\sigma's)(\rho(y)\circ_{\n}\rho(z)-\rho(y\circ_{\g} z))\nonumber\\
             &\ \ +(r\sigma-\sigma's)(\rho(y)\circ_{\n}\rho(x)-\rho(y\circ_{\g} x))\bcr s\rho(z)\nonumber\\
             &\ \ -(r\sigma-\sigma's)(\rho(y)\circ_{\n}\rho(x\circ_{\g} z)-\rho(y\circ_{\g} (x\circ_{\g} z)))-c.p.(x,y)\}\nonumber\\
             &\ \ -(r\sigma-\sigma's)(\rho([x,y])\circ_{\n}\rho(z)-\rho([x,y]\circ_{\g} z)).
         \end{align}
         The first equation holds because ${r|}_V=id_V$; the second one is guaranteed by the fact that $r,s$ respect actions of crossed modules. For example, the first term in the second equation is obtained as follows:
         \begin{align*}
             r(\rho(x)\bcl\beta(y,z))-\rho'(x)\bcl\beta(y,z)
             &=s\rho(x)\bcl r(\sigma(\rho(y)\circ_{\n}\rho(z)-\rho(y\circ_{\g}z)))\\
             &\ \ \ \ -s\rho(x)\bcl\sigma'(s\rho(y)\circ_{\n}s\rho(z)-s\rho(y\circ_{\g}z))\\
             &=s\rho(x)\bcl(r\sigma-\sigma's)(\rho(y)\circ_{\n}\rho(z)-\rho(y\circ_{\g} z)).
         \end{align*}

         Now, define $\phi\in\Hom(\g\otimes\g,\m)$ by $\phi(x,y):=(r\sigma-\sigma's)(\rho(x)\circ_{\n}\rho(y)-\rho(x\circ_{\g} y))$, then direct computation shows that $\mu'\phi=0$, hence $\phi\in\Hom(\g\otimes\g,V)$. It follows from \eqref{theta-theta'} that $\theta-\theta'=d\phi$, hence
         $[\theta-\theta']=0$, which completes the proof. 
\end{proof}

\begin{lemma}\label{2}
    The map $S: H^3(\g,V)\to CExt(\g,V)$ constructed in the proof of Theorem \ref{m th} is well-defined.
\end{lemma} 
\begin{proof}
We claim that $S$ does not depend on the choice of $\tal$ and that $3$-cocycles in the same cohomology class lead to equivalent crossed module extensions. 

Indeed, suppose that there is another $\tal'\in C^2(F,V)$ such that $d\tal'=\pi^*\theta$, then we get a crossed module extension through it, see \eqref{pre lie m} and \eqref{tal action}. We will show that this crossed module extension is equivalent to the one defined by $\tal$. By $d(\tal-\tal')=0$ and $H^2(F,V)=0$, there exists $k\in C^1(F,V)$ such that $\tal-\tal'=dk$. It follows that $\alpha-\alpha'=i^*(\tal-\tal')=i^*dk=di^*k$. Define $\phi:\m\to\m$ by $\phi(Y+v)=Y+v+i^*k(Y)$ for $Y\in K,v\in V$. Then it is an isomorphism of pre-Lie algebras inducing the identities on $V$ and $K$, and so the (elementary) equivalence of crossed module extension \eqref{eq 5}. Therefore, $S$ is independent of the choice of $\tal$.

Suppose there are $3$-cocyles $\xi',\xi$ such that $[\xi']=[\xi]$. Then $\xi'-\xi=dl$ for some $l\in C^2(\g,V)$. Hence $\pi^*\xi'=\pi^*\xi+\pi^*dl=d\tal+d\pi^*l=d(\tal+\pi^*l)$. Since $S$ (or say the equivalence class of the resulting crossed module extension) is independent of the choice of $\tal'$ that satisfys $d\tal'=\pi^*\xi'$, we can choose $\tal'=\tal+\pi^*l$. This forces $\alpha'=i^*\tal'=i^*\tal+i^*\pi^*l=\alpha$, thus they correspond to the same extension \eqref{eq 4}. Now it is obvious that identity gives the (elementary) equivalence of crossed module extensions \eqref{eq 5} and we conclude that $S$ is well-defined.
\end{proof}

\begin{lemma}\label{3}
$ST=1_{CExt}$. 
\end{lemma}
\begin{proof}
$\bullet$ We first recall briefly how to construct $T:CExt\to H^3$ and $S:H^3\to CExt$.
Given a crossed module extension:
\begin{equation}\label{cm 1}
    0\to V\xrightarrow{i}\m\xrightarrow{\mu}\n\xrightarrow{p}\g\to 0,
\end{equation}
where the action of $\n$ on $\m$ is denoted by $(\bcl,\bcr)$, 
we can choose sections $\rho$ of $p$ and $\sigma$ of $\mu$ to define a $3$-cocycle $\theta$ (see \eqref{theta1}). Then $T$ maps the equivalence class of \eqref{cm 1} to $[\theta]$. Let $F$ be the free pre-Lie algebra generated by $\g$ and $\pi: F\to\g$ be the projection. For this $3$-cocycle $\theta$, there exists $\tal_0\in C^2(F,V)$ such that $d\tal_0=\pi^*\theta$. By this $\tal_0$ we can construct a crossed module extension and $S$ maps $[\theta]$ to equivalence class of it.
Our goal is to show that the last crossed module extension is equivalent to the original one \eqref{cm 1}.  

\item Observing that $d\tal_0=\pi^*\theta$ only depends on its values on $\pi(F)$, and that $\pi s\pi=\pi$ for any section $s$ of $\pi$, we have
\begin{equation}\label{tal on g}
    d\tal_0(X,Y,Z)=d\tal_0(s\pi X,s\pi Y,s\pi Z),\qquad X,Y,Z\in F.
\end{equation}
Write $F=\g\oplus ker\pi$ and choose section $s$ to be the inclusion, i.e., $sx=(x,0)$ for $x\in\g$. Then by \eqref{tal on g}, 
$\pi^*\theta(x+X,y+Y,z+Z)=d\tal_0(sx,sy,sz)$ for $x+X,y+Y,z+Z\in\g\oplus ker\pi$, so $\pi^*\theta$ is completely depended on the restriction of $\tal_0$ on ${(\g,0)\otimes(\g,0)}$. Since $S$ is indenpendent of the choice of $\tal$ that satisfys $d\tal=\pi^*\theta$, we can adjust the values of $\tal_0$ in other parts, like $(0,ker\pi)\otimes(0,ker\pi)$, and define a new $\tal$ which is easier to work with.

Indeed, since $F$ is free on $\g$, the section $\rho:\g\to\n$ uniquely induces a pre-Lie homomorphism $G:F\to \n$ such that $G(x,0)=\rho x$ for $x\in\g$. We can define $\tal\in C^2(F,V)$ by
        \begin{align*}
\tal(x+X,y+Y)&:=\tal_0(x,y)+f(x,Y)+f(X,y)+f(X,Y),
        \end{align*}
        where $if(X,Y):=\sigma(GX)\circ\sigma(GY)-\sigma(GX\circ GY)$ for $X,Y\in F$.
        Now, we get a crossed module extensions as in the proof of Theorem \ref{m th}:
        \begin{equation}\label{cm 2}
            0\to V\to ker\pi\oplus_{\tal} V\xrightarrow{\mu_{\tal}}F\xrightarrow{\pi}\g\to 0.
        \end{equation}
        The pre-Lie structure on $ker\pi\oplus V$ is given by \eqref{pre lie m}, and the action of $F$ on $ker\pi\oplus V$, denoted also by $(\bcl,\bcr)$, is given by \eqref{tal action}.

        $\bullet$ To show that crossed module \eqref{cm 2} is equivalent to \eqref{cm 1}, we need to define $H:ker\pi\oplus_{\tal}V\to\m$ appropriately to give the (elementary) equivalence:
        \begin{equation}
\begin{array}{ccccccccccc}
0&\longrightarrow& V&\longrightarrow&ker\pi\oplus_{\tal}V&\stackrel{\mu_{\tal}}\longrightarrow&F&\stackrel{\pi}\longrightarrow&\g&\longrightarrow&0
\\
 &            &\Big\|&       &{\scriptstyle H}\Big\downarrow&          &{\scriptstyle G}\Big\downarrow& &\Big\|& & \\
 0&\longrightarrow&V&\stackrel{i}\longrightarrow&\m&\stackrel{\mu}\longrightarrow&\n&\stackrel{p}\longrightarrow&\g&\longrightarrow&0.
 \end{array}\end{equation}
        %\\
%\begin{center}
%\begin{tikzpicture}
 % \matrix (m) [matrix of math nodes, row sep=3em, column sep=3em]
 %   { 0 & V  & ker\pi\oplus_{\tal}V & F & \g & 0 \\
 %     0 & V & \m & \n & \g & 0. \\ };
 % { [start chain] \chainin (m-1-1);
 %   \chainin (m-1-2);
 %   { [start branch=V] \chainin (m-2-2)
 %       [join={node[right,labeled] {id}}];}
 %   \chainin (m-1-3);
 %   { [start branch=kerpioplus_{tal}V] \chainin (m-2-3)
  %      [join={node[right,labeled] {H}}];}
  %  \chainin (m-1-4) [join={node[above,labeled] {\mu_{\tal}}}];
   % { [start branch=F] \chainin (m-2-4)
    %    [join={node[right,labeled] {G}}];}
    %\chainin (m-1-5) [join={node[above,labeled] {\pi}}]; 
    %{ [start branch=g] \chainin (m-1-6);}
  %\chainin (m-2-5) [join={node[right,labeled] {id}}];}
 % { [start chain] \chainin (m-2-1);
   % \chainin (m-2-2);
   % \chainin (m-2-3) [join={node[above,labeled] {i}}];
   % \chainin (m-2-4) [join={node[above,labeled] {\mu}}];
   % \chainin (m-2-5) [join={node[above,labeled] {p}}]; 
   % \chainin (m-2-6);}
%\end{tikzpicture}
%\end{center}
         The commutativity of the diagram indicates that we shall define $H:ker\pi\oplus_{\tal}V\to\m$ by $H(Y+v)=\sigma GY+iv$.
        $H$ is a pre-Lie homomorphism because $i\tal(X,Y)=\sigma(GX)\circ_{\m}\sigma(GY)-\sigma(GX\circ_{\n} GY)$ for $X,Y\in ker\pi$.
        
        We next show that $H$ and $G$ respect the actions of crossed modules. Take $H(X\bcl (Y+v))=GX\bcl H(Y+v), X\in F,Y+v\in ker\pi\oplus V$ as an example. This is equivalent to:
        \begin{align*}
             \sigma G(X\circ_F Y)+i\tal(X,Y)+i(X\circ_l v)&=GX\bcl\sigma GY+GX\bcl iv\\
             &=\mu\sigma GX\bcl\sigma GY+GX\bcl iv\\
             &=\sigma GX\circ_{\m}\sigma GY+GX\bcl iv.
       \end{align*} 
Since \begin{align*}
    GX\bcl iv-i(X\circ_lv)&=GX\bcl iv-i(\pi X\circ_lv)\\
    &=GX\bcl iv-\rho\pi X\bcl iv\\
    &=(G-\rho\pi)X\bcl iv,
\end{align*}
we have $i\tal(X,Y)=\sigma(GX)\circ_{\m}\sigma(GY)-\sigma(GX\circ_{\n} GY)+(G-\rho\pi)X\bcl v$.
       We claim that $(G-\rho\pi)X\circ v=0$ for all $X\in F,v\in V$, hence $H,G$ respect the actions by the definition of $\tal$. For $X\in ker\pi$, $(G-\rho\pi)X=GX\in ker\ p=im\ \mu$. Thus $GX=\mu m$ for some $m\in\m$, and $GX\circ v=\mu m\circ i(v)=m\circ i(v)=m\circ\mu i(v)=0$.
       For $X=(x,0)\in\g\oplus 0$, we have $(G-\rho\pi)X=G(x,0)-\rho x=0$.     
       
       $\bullet$ To sum up, $ST$ maps equivalence class of crossed module \eqref{cm 1} to $[\theta]$, and then to equivalence class of crossed module \eqref{cm 2}, which equals to the original one, hence $ST=1_{CExt}$.
       \end{proof}

       \begin{lemma}\label{4}
           $TS=1_{H^3(\g,V)}$
       \end{lemma}
       \begin{proof}
       We begin by recalling the construction of $TS$. Given a $[\theta]\in H^3(\g,V)$, we can find $\tal\in C^2(F,V)$ such that $\pi^*\theta=d\tal$, then we define a crossed module \eqref{cm 2} by this $\tal$. We next choose sections $s$ of $\pi$ and $\sigma_{\tal}$ of $\mu_{\tal}$ to get a map $\beta\in C^2(\g,ker\pi\oplus_{\tal}V)$, which defines a new $3$-cocycle $\theta'$. Our goal is to show $[\theta]=[\theta']$. In fact, $\theta'-\theta=d(s^*\tal)$. We only need to prove $\pi^*\theta'-\pi^*\theta=\pi^*d(s^*\tal)$ since $\pi^*$ is injective.
        
        In particular, we can take $\sigma_{\tal}: ker\pi\to ker\pi\oplus_{\tal}V$ to be the inclusion, i.e., $\sigma_{\tal}(X)=(X,0)$ for $X\in ker\pi$. 
        Then $\beta(x,y)=(sx\circ sy-s(x\circ y),0)$ for $x,y\in\g$. 
        Denote the projection of $ker\pi\oplus_{\tal}V$ onto $ker\pi$ by $pr$. Then by \eqref{theta1} and \eqref{tal action} we have
        \begin{align*}
         \pi^*\theta'(X,Y,Z)
           &=s\pi X\circ pr\beta(\pi Y,\pi Z)+pr\beta(\pi Y,\pi X)\circ s\pi Z\\
           &\ \ -\beta(\pi Y,\pi X\circ \pi Z)-\beta(\pi Y\circ \pi X, \pi Z)\\
           &\ \ +\tal(s\pi X,pr\beta(\pi Y,\pi Z))+\tal(pr\beta(\pi Y,\pi X),s\pi Z)-c.p.(X,Y)\\
           &=\tal(s\pi X,pr\beta(\pi Y,\pi Z))+\tal(pr\beta(\pi Y,\pi X),s\pi Z)-c.p.(X,Y),\\
           &=\tal(s\pi X,s\pi Y\circ s\pi Z-s(\pi Y\circ\pi Z))\\
           &\ \ +\tal(s\pi Y\circ s\pi X-s(\pi Y\circ\pi X),s\pi Z)-c.p.(X,Y),
        \end{align*}
        the second equation holds since the image of $\pi^*\theta'$ is in $V$.
        By \eqref{tal on g} we have
        \begin{align*}
           \pi^*\theta(X,Y,Z)&=d\tal(s\pi X,s\pi Y,s\pi Z),\\
           &=s\pi X\circ\tal(s\pi Y,s\pi Z)+\tal(s\pi Y,s \pi X)\circ s\pi Z\\
           &\ \ -\tal(s\pi Y,s\pi X\circ s\pi Z)-\tal(s\pi Y\circ s\pi X, s\pi Z)-c.p.(X,Y),
        \end{align*} 
        for $X,Y,Z\in F$. So
        \begin{align*}
            \pi^*(\theta-\theta')(X,Y,Z)&=s\pi X\circ\tal(s\pi Y,s\pi Z)+\tal(s\pi Y,s \pi X)\circ s\pi Z\\
           &-\tal(s\pi Y,s(\pi X\circ \pi Z))-\tal(s(\pi Y\circ \pi X), s\pi Z)-c.p.(X,Y)\\
           &=\pi^*d(s^*\tal)(X,Y,Z)
       \end{align*}
        as required.
    \end{proof}
    $\bullet$ Relation with abelian extensions of the Lie algebra.
    
Let $A$ be a Lie algebra, $W$ be a $A$-module. An abelian extension of $A$ by $W$ is an exact sequence of Lie algebras:
\[0\to W\to \hat{A}\to A\to 0,\]
where all arrows are homomorphisms of Lie algebras.  Two abelian extension of $A$ by $W$ are equivalent if there exists a homomorphism $\gamma$ such that the following diagram commutes:
\begin{equation*}
\begin{array}{ccccccccc}
0&\longrightarrow& W&\longrightarrow&\hat{A}_2&\longrightarrow&A&\longrightarrow&0\\
 &            &\Big\|&       &{\scriptstyle \gamma}\Big\downarrow&          &\Big\|& &\\
 0&\longrightarrow&W&\longrightarrow&\hat{A}_1&\longrightarrow&A&\longrightarrow&0.
 \end{array}
 \end{equation*}
 Note that in this case $\gamma$ must be an isomorphism.
Denote by $AExt(A,W)$ the set of equivalence classes of abelian extensions of $A$ by $W$. 

Let $\g$ be a pre-Lie algebra, and $V$ be a $\g$-module. Then by remark $2.11$ of \cite{GLST}, there is a $\g^c$-module structure on $\Hom (\g,V)$ given by
\[(x\tr f)(y):=x\circ f(y)+f(x)\circ y-f(x\circ y),\qquad \forall f\in\Hom(\g,V),\ x,y\in\g.\]
    \begin{corollary}\label{cor of m th}
        There is a one-to-one correspondence between $CExt(\g,V)$ and $AExt(\g^c,\Hom(\g,V))$.
    \end{corollary}
    \begin{proof}
        By \cite{GLST}, there is an isomorphism $\phi: H^n_{preLie}(\g,V)\to H^{n-1}_{Lie}(\g^c,\Hom(\g,V))$ induced by the cochain map $\phi:\Hom(\wedge^{n-1}\g\otimes\g,V)\to\Hom(\wedge^{n-1}\g^c,\Hom(\g,V))$, which is defined by 
\[(\phi f)(x_1,\cdots,x_{n-1})(x_n):=f(x_1,\cdots,x_n).\]
        Conbine it with theorem \ref{m th} and the well-known bijection between $ H^2_{Lie}(\g^c,\Hom(\g,V))$ and $AExt(\g^c,\Hom(\g,V))$ yields the desire correspondence:
        \begin{equation*}
\begin{array}{ccc}
H^3_{preLie}(\g,V)&\stackrel{1-1}\longrightarrow& H^2_{Lie}(\g^c,\Hom(\g,V))\\
{\scriptstyle 1-1}\Big\uparrow&       &{\scriptstyle 1-1}\Big\downarrow\\
 CExt(\g,V)&\stackrel{Corollary\  \ref{cor of m th}}\longrightarrow&AExt(\g^c,\Hom(\g,V)).
 \end{array}
 \end{equation*}
  %      \\
 %\begin{center}
     %\begin{tikzpicture}
          %\matrix (m) [matrix of math nodes, row sep=3em, column sep=3em]
     %{ H^3_{preLie}(\g,V) & H^2_{Lie}(\g^c,\Hom(\g,V))\\
     %CExt(\g,V) & AExt(\g^c,\Hom(\g,V)).\\};
     %{ [start chain] \chainin (m-1-1);
     %\chainin (m-1-2) [join={node [above, labeled] {1-1}}];
     %\chainin (m-2-2)                [join={node[right,labeled] {1-1}}];
       %  }
     %{ [start chain] \chainin (m-2-1);
     %\chainin (m-1-1) 
     %[join={node[left,labeled] {1-1}}];}
     %{ [start chain] \chainin (m-2-1);
     %\chainin (m-2-2) 
     %[join={node[below,labeled] {Corollary\  \ref{cor of m th}}}];
     %\end{tikzpicture}
 %\end{center}
    \end{proof}

\section{Some relations on crossed modules of pre-Lie algebras}
%We use notations .

In this section, we summarize some relations between crossed modules of pre-Lie algebras and crossed modules of other algebras, which are similar to the relations between these algebras.

Denote the commutator Lie algebra of any associative algebra $\m$ by $\m^c$.
Let $(\m\xrightarrow{\mu}\n,\cdot_l,\cdot_r)$ be a crossed module of associative algebras, then as mentioned in \cite{DIKL}, $(\m^c\xrightarrow{\mu}\n^c,\triangleright)$ is a crossed module of Lie algebras with the action $\triangleright: \n^c\wedge\m^c\to\m^c$ given by $x\triangleright u:=x\cdot_l u-u\cdot_r x$ for $x\in\n^c,u\in\m^c$. 
\begin{proposition}[\cite{S}, proposition $4.9$]\label{S pro 4.9}
    Let $((\m,\circ_m)\xrightarrow{\mu}(\n,\circ_n),\circ_l,\circ_r)$ be a crossed module of pre-Lie algebras and $\m^c,\n^c$ be their commutator Lie algebras. Then $(\m^c\xrightarrow{\mu}\n^c,\triangleright)$ is a crossed module of Lie algebras, where $x\triangleright u:=x\circ_l u-u\circ_r x$ for $x\in\n,u\in\m$.
\end{proposition}
\begin{proposition}[\cite{ZL}, proposition $4.8$]\label{ZL pro 4.8}
    Let $\big((\m,[\cdot,\cdot]_m,T_m)\xrightarrow{\mu}(\n,[\cdot,\cdot]_n,T_n),\rho\big)$ be a crossed module of Rota-Baxter Lie algebras.  Define $\circ_m:\m\otimes\m\to\m$, $\circ_n:\n\otimes\n\to\n$, and $\circ_l:\n\otimes\m\to\m$, $\circ_r:\m\otimes\n\to\m$ by
    \begin{align*}
        x\circ_n y&:=[T_nx,y]_n,\qquad
        u\circ_m v:=[T_mu,v]_m,\\     
        x\circ_l u&:=\rho(T_nx)u,\qquad
        u\circ_r x:=-\rho(x)(T_mu),
    \end{align*}
    for $x,y\in\n,u,v\in\m$.
    Then $\big((\m,\circ_m)\xrightarrow{\mu}(\n,\circ_n),\circ_l,\circ_r\big)$ is a crossed module of pre-Lie algebras.
\end{proposition}
The notion of crossed module of dendriform algebras is given in \cite{Das dend}, see definition $6.6$ and also remark $2.5$.
It gives rise to a crossed module of associative algebra, see subsection $6.6$ of \cite{Das dend}. The following proposition shows that it also gives a crossed module of pre-Lie algebras.
\begin{proposition}\label{dend pl}
   Let $\big((\m,\succ_m,\prec_m)\xrightarrow{\mu}(\n,\succ_n,\prec_n),\succ,\prec\big)$ be a crossed module of dendriform algebras. Define $\circ_m:\m\otimes\m\to\m$, $\circ_n:\n\otimes\n\to\n$, and $\circ_l:\n\otimes\m\to\m$, $\circ_r:\m\otimes\n\to\m$ by
    \begin{align*}
        x\circ_n y&:=x\succ_n y-y\prec_n x,\qquad
        u\circ_m v:=u\succ_m v-v\prec_m u,\\     
        x\circ_l u&:=x\succ u-u\prec x,\qquad
        u\circ_r x:=u\succ x-x\prec u,
    \end{align*}
    for $x,y\in\n,u,v\in\m$.
    Then $\big((\m,\circ_m)\xrightarrow{\mu}(\n,\circ_n),\circ_l,\circ_r\big)$ is a crossed module of pre-Lie algebras.
\end{proposition}
\begin{proof}
    By proposition $5.2$ of \cite{B}, $(\m,\circ_m)$ and $(\n,\circ_n)$ are pre-Lie algebras. Since $\mu$ is a dendriform homomorphism, it is also a pre-Lie homomorphism. The fact that $(\succ,\prec)$ is an action of dendriform algebras implies that $(\circ_l,\circ_r)$ is an action of pre-Lie algebras. Moreover,  
    $\mu(x\circ_l u)=\mu(x\succ u-u\prec x)
        =x\succ_n\mu u-\mu u\prec_n x=x\circ_n\mu u$ for $x\in\n,u\in\m$. Similar considerations show that $\big((\m,\circ_m),(\n,\circ_n),\mu,\circ\big)$ is a crossed module of pre-Lie algebras.
\end{proof}
We summarize these relations in the following diagram:\\
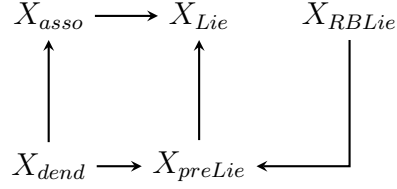
\begin{center}
\begin{tikzpicture}
    [node distance=2cm]
    \node (Xasso) {$X_{asso}$};
    \node (XLie) [right of=Xasso] {$X_{Lie}$};
    \node (XRBLie) [right of=XLie] {$X_{RBLie}$};
    \node (XpreLie) [below of=XLie] {$X_{preLie}$};
    \node (Xdend) [left of=XpreLie] {$X_{dend}$};
    \draw [arrow] (Xasso) -- (XLie);
    \draw [arrow] (Xdend) -- (Xasso);
    \draw [arrow] (XpreLie) -- (XLie);
    \draw [arrow] (XRBLie) |- (XpreLie);
    \draw [arrow] (Xdend) -- (XpreLie);
\end{tikzpicture}
\end{center}
where $X_{asso},X_{Lie},X_{preLie},X_{RBLie},X_{dend}$ are abbreviations of crossed modules of associative, Lie, pre-Lie, Rota-Baxter Lie and dendriform algebras respectively, and the arrow means that we can obtain one crossed module from the other.

\end{document}